\documentclass[12pt]{amsart}
\usepackage[utf8]{inputenc}
\usepackage[hidelinks]{hyperref}
\usepackage[T1]{fontenc}

\usepackage{amsmath}
\usepackage[utf8]{inputenc}
\usepackage{amssymb}
\usepackage{amsthm}
\usepackage{graphicx}
\usepackage{color, soul}
\usepackage{centernot}
\usepackage{verbatim}
\usepackage{multirow}
\usepackage{array}   
\newcolumntype{L}{>{$}l<{$}} 
\usepackage{setspace}
\usepackage{mathrsfs}
\usepackage{enumitem}
\usepackage{tikz}
\usepackage{bigints}
\usepackage{bm}
\usepackage{caption}
\usepackage{subcaption}
\usepackage{url}
\usepackage{imakeidx}
\makeindex[columns=2, intoc]
\usepackage{tikz-cd}

\makeatletter
\renewcommand{\@biblabel}[1]{[#1]\hfill}
\makeatother

\newtheorem{theorem}{Theorem}[section]

\newtheorem{lemma}[theorem]{Lemma}
\newtheorem{proposition}[theorem]{Proposition}

\theoremstyle{definition}
\newtheorem{definition}[theorem]{Definition}
\newtheorem{remark}[theorem]{Remark}

\usepackage{mathtools}
\mathtoolsset{showonlyrefs}

\newcommand{\OO}{\mathcal{O}}

\newcommand{\QQ}{\mathbb{Q}}
\newcommand{\ZZ}{\mathbb{Z}}

\newcommand{\FF}{\mathbb{F}}

\DeclareMathOperator{\End}{End}

\DeclareMathOperator{\disc}{disc}

\DeclareMathOperator{\Sel}{Sel}
\DeclareMathOperator{\Hess}{Hess}

\DeclareFontFamily{U}{wncy}{}
\DeclareFontShape{U}{wncy}{m}{n}{<->wncyr10}{}
\DeclareSymbolFont{mcy}{U}{wncy}{m}{n}
\DeclareMathSymbol{\Sha}{\mathord}{mcy}{"58}

\newcommand{\pp}{\mathfrak{p}}

\newcommand{\qq}{\mathfrak{q}}

\renewcommand{\bar}{\overline}

\usepackage{geometry}
\title{Constructing families of 3-Selmer companions}

\author{Harry Spencer}
\address{University College London, WC1H 0AY, UK}
\email{harry.spencer.22@ucl.ac.uk}

\begin{document}

\begin{abstract}
Mazur and Rubin introduced the notion of $n$-Selmer companion elliptic curves and gave several examples of pairs of non-isogenous Selmer companions. We construct several pairs of families of elliptic curves, each parameterised by $t\in\ZZ$, such that the two curves in a pair corresponding to a given $t$ are non-isogenous $3$-Selmer companions, possibly provided that $t$ satisfies a simple congruence condition.
\end{abstract}

\maketitle

\section{Introduction}
We recall the following definition introduced by Mazur and Rubin.

\begin{definition}[{=\cite[Definition 1.2]{MazurRubin2012}}]\label{defn:pSelm}
    Let $n>1$ be an integer. Two elliptic curves $E$ and $H$ over a number field $k$ are $n$\emph{-Selmer companions} if
    \begin{equation}\label{eqn:1}
        \Sel_n(E^\chi/k) \cong \Sel_n(H^\chi/k)
    \end{equation}
    for all quadratic characters $\chi$ of $k$.
\end{definition}

This definition was introduced in order to study the question of which arithmetic data associated to an elliptic curve $E$ can be recovered from the function $\chi\mapsto \dim_{\FF_p} \Sel_p(E^\chi)$ on quadratic characters, for some fixed prime $p$. The existence of non-isogenous $p$-Selmer companions shows that the isogeny class of $E$ cannot necessarily be recovered from such a function. The question of which data can be recovered remains open, although Mazur and Rubin conjecture that having isomorphic $n$-torsion is a necessary condition for $n$-Selmer (near-)companionship (\cite[Conjecture 7.15]{MazurRubin2012}) and work of Yu (\cite{Yu2Companions}) establishes this conjecture in the case $n=2$.

Several examples of non-isogenous Selmer companions are given in \cite{MazurRubin2012} and---to the best of the author's knowledge---no others appear in the literature. In this note we construct infinitely many non-isogenous pairs of $3$-Selmer companions. In particular, we study the following two-parameter families of curves:
\begin{eqnarray}
    E_{q,t} : \hspace{0.2cm} y^2 &=& x^3 + q^2x^2 + 3q^3x +3q^6t \\
    H_{q,t}: \hspace{0.2cm} y^2 &=& x^3+q(27-2q-81qt)x^2+q(q-9)^3x,
\end{eqnarray}
where $H_{q,t}$ is the Hessian curve associated to this model of $E_{q,t}$. In \S4 we show that several fixed values of $q$ yield families of $3$-Selmer companions as $t$ varies, possibly given a congruence condition on $t$:

\begin{theorem}[=Theorem \ref{thm:main} + Proposition \ref{prop:isog}]
    Suppose $(q,t)\in\ZZ^2$ are as in one of the columns of Table \ref{tab:1} below. Then $E_{q,t}$ and $H_{q,t}$ are non-isogenous $3$-Selmer companions over every number field, provided $(q,t)\notin \{ (1,1), (1,9), (3,-1), (8,0), (12,0) \}.$
\end{theorem}

\renewcommand{\arraystretch}{1.5}
\begin{table}[ht]
    \centering
    \begin{tabular}{|c|c|c|c|}
         \hline
         Value of $q$ & $-15,1,5,13,17,21$ & $3,7,11$ & $8,10,12$ \\
         \hline
         Condition on $t$ & $t\equiv1\pmod{8}$ & $t\equiv3\pmod{4}$ & $t\in\ZZ$ \\
         \hline
    \end{tabular}
    \caption{Parameters $q,t$ for which $E_{q,t}$ and $H_{q,t}$ are $3$-Selmer companions}
    \label{tab:1}
\end{table}

The main idea of the proof is to simplify a set of sufficient conditions for $3$-Selmer companionship (\cite[Theorem 3.1]{MazurRubin2012}, cf.\ Proposition \ref{prop:3Comps}) in the case of an elliptic curve and its Hessian. We then apply this to the families $E_{q,t}$ and $H_{q,t}$ for certain values of $q$. Note that we do not claim that the pairs $(q,t)$ appearing in Table \ref{tab:1} constitute an exhaustive list in any sense (see Remark \ref{rmk:exhaust}). Moreover we note that, by tweaking the family $E_{q,t}$, one can likely construct many other examples.

We hope that the techniques and results of this work may be helpful for testing potential necessary conditions; for example, Mazur and Rubin speculate that having equal conductors may be a necessary condition and indeed this is the case above.

\section{Sufficient conditions}
Here we recall the main result of \cite{MazurRubin2012}, a set of sufficient conditions (stated below for $q=3$) for $q$-Selmer companionship, where $q$ is a prime power.

\begin{definition}\label{defn:TateGroup}
    Let $E/K$ be an elliptic curve with potentially multiplicative reduction over a local field of mixed characteristic. Given its associated Tate parameter $q\in K^\times$ with isomorphism $\tau_{E/K}: \bar{K}^\times/q^\ZZ \to E(\bar{K})$ and a positive integer $m$, we define the canonical subgroup $\mathcal{C}_{E/K}[m]\subset E[m]$ to be $\tau_{E/K}(\mu_m)$.
\end{definition}

\begin{remark} \label{rmk:canonical}
    An alternative description of the canonical subgroup of an elliptic curve with potentially multiplicative reduction is as $E_0(K')[m]$, the $m$-torsion in the subgroup of $E(K')$ consisting of points reducing to smooth points on the special fibre of a minimal model, where $K'=K(E[m])$ (e.g.\ \cite[V.4]{SilvermanAdvanced}). In particular, given a (necessarily minimal) model $\mathcal{E}: y^2+xy=x^3+a_2x^2+a_4x+a_6$ with $a_2,a_4,a_6$ of positive valuation, the non-identity points of $\mathcal{C}_{E/K}[m]$ are those $m$-torsion points with $(x,y)$-coordinates of non-positive valuation.
\end{remark}

\begin{theorem}[{=\cite[Theorem 3.1 + Remark 3.3]{MazurRubin2012}}]\label{thm:MazurRubin}
    Let $E$ and $H$ be elliptic curves over a number field $k$, and write $S_1$ and $S_2$ for the set of primes of potentially multiplicative reduction of $E$ and of $H$, respectively. Suppose the following:
    \begin{itemize}
        \item there is a $G_k$-isomorphism $\alpha: E[3]\xrightarrow{\sim} H[3]$;
        \item $S_1=S_2$;
        \item for every $\pp$ above $3$, $\pp\in S_1=S_2$;
        \item for every $\qq\in S_1=S_2$, we have $\alpha(\mathcal{C}_{E/k_\qq}[3])=\mathcal{C}_{H/k_\qq}[3]$;
        \item neither $E$ nor $H$ has any prime of additive reduction with Kodaira type \emph{II}, \emph{II}$^*$, \emph{IV} or \emph{IV}$^*$. 
    \end{itemize}
    Then $E$ and $H$ are $3$-Selmer companions over every finite extension of $k$.
\end{theorem}

Each of these conditions is simple to check for a given putative pair of companions except the fourth; in the following section we show that this fourth condition holds for a class of examples arising from the Hessian construction. 

\section{Hessian curves}

As mentioned, the main idea behind the proof of Theorem \ref{thm:main} is to simplify the conditions of Theorem \ref{thm:MazurRubin} when $H$ is the Hessian curve associated to a Weierstrass model of $E$. The Hessian construction is well-known and is described in detail, for example, in \cite[\S2.2]{MPT24}. In this section we recap the construction, relate the reduction behaviour of an elliptic curve with that of its associated Hessian curve, and show how it links to the fourth condition of Theorem \ref{thm:MazurRubin}.

\begin{definition}
     For an elliptic curve $E: F(X_1,X_2,X_3)=0$ given by the vanishing of a cubic homogeneous polynomial in three variables, the corresponding \textit{Hessian curve}, $\Hess(E)$ is defined by the vanishing of the determinant of the Hessian matrix $(\frac{\partial F}{\partial X_i \partial X_j})_{i,j}$.
\end{definition}

\begin{remark}
    Given an elliptic curve in Weierstrass form, the isomorphism class of the associated Hessian curve is invariant under change of Weierstrass equation (see \cite[\S2.2]{MPT24}).
\end{remark}

The following lemma, which relates the reduction behaviour of an elliptic curve with that of its Hessian, will allow us to verify some of the conditions of Theorem \ref{thm:MazurRubin} in many cases. 

\begin{lemma}\label{lem:mult}
    Let $K$ be a local field of mixed characteristic. Let $E/K$ be an elliptic curve in Weierstrass form and suppose the associated Hessian curve $H/K$ is not a union of three lines. We have:
    \begin{enumerate}
        \item If $E$ has potentially multiplicative reduction, then so does $H$.
        \item If $E$ has multiplicative reduction, then so does $H$. Moreover, in this case, $E$ has split multiplicative reduction if and only if $H$ does also.
    \end{enumerate} 
\end{lemma}

\begin{remark}
    The condition on the Hessian curve not being a union of three lines is equivalent to $j(E)\ne0$.
\end{remark}

\begin{proof}[Proof of Lemma \ref{lem:mult}]
    First suppose that $E$ has split multiplicative reduction. We can take $E$ to be given by a (minimal) Tate curve model $\mathcal{E}/\OO_K: y^2+xy=x^3+a_4x+a_6$, where $a_4$ and $a_6$ have positive valuation (see \cite[V.3]{SilvermanAdvanced}). Computing the Hessian gives a model
    \begin{equation}\label{eqn:Hess}\tag{$\dagger$} \mathcal{H}/\OO_K : y^2 + xy = -(3+3^2\cdot32a_4(24a_4-1))x^3+3^2\cdot(196a_4^2-4a_4-48a_6)x^2 + \hdots \end{equation}
    after a change of coordinates $Z'=Z+12X$. Note that the coefficient of $x$ and the constant term in \eqref{eqn:Hess} are $a_4-144a_4^2+72a_6$ and $4a_4^2-3a_6$, respectively. Setting $u=-(3+3^2\cdot32a_4(24a_4-1))$ and $(X',Y')=(uX,uY)$, then completing the cube, we obtain a model given by a Tate curve and see that $H$ has split multiplicative reduction.

    To conclude, suppose that $E$ has potentially multiplicative reduction. There is a finite extension of $K$ over which $E$ acquires split multiplicative reduction, and over this extension $H$ must also have split multiplicative reduction by the above.
\end{proof}

\begin{remark}
    The converse of Lemma \ref{lem:mult} is false; there are many examples of elliptic curves with (potentially) good reduction for which the associated Hessian curves have multiplicative reduction.
\end{remark}

One of the main drivers of interest in Hessians of curves is that they keep track of points of inflection. For elliptic curves, this gives an isomorphism on $3$-torsion between an elliptic curve and its associated Hessian curve.

\begin{proposition}\label{prop:tors_isom}
    Let $E/k$ be an elliptic curve over a number field given by Weierstrass equation whose Hessian curve $H=\Hess(E)$ is not the union of three lines. There is a $G_k$-module isomorphism $E[3]\cong H[3]$.
\end{proposition}

\begin{proof}
    See \cite[Remark 2.4]{MPT24}. The key point is that the intersection points of $E$ and $H$ are precisely the inflection points of both curves, and hence are the $3$-torsion points of both curves.
\end{proof}

Using this identification, we can verify the fourth condition of Theorem \ref{thm:MazurRubin} for a curve and its Hessian.

\begin{proposition}\label{prop:canon}
    Let $K$ be a local field of mixed characteristic with residue characteristic $p$. Consider a Weierstrass model for an elliptic curve $E/K$ such that both $E$ and $H=\Hess(E)$ have potentially multiplicative reduction. The identification $\alpha: E[3]\to H[3]$ as in Proposition \ref{prop:tors_isom} satisfies
    \[\alpha(\mathcal{C}_{E/K}[3])=\mathcal{C}_{H/K}[3].\]
\end{proposition}

\begin{proof}
    By passing to a suitable extension, we may assume that $E$ has split multiplicative reduction and note that Lemma \ref{lem:mult} ensures that $H$ does also.

    Take a model $\mathcal{E}$ for $E$ given by a Tate curve and the corresponding model $\mathcal{H}$ for its Hessian given by \eqref{eqn:Hess}, as in the proof of Lemma \ref{lem:mult}. We again apply the change of coordinates $(X',Y',Z')=(uX,uY,Z+12X)$, where $u=-(3+3^2\cdot32a_4(24a_4-1))$. 
    
    Write $v$ for the valuation on $K$, normalised so that $v(p)=1$ and choose a non-identity point $P=(X_0:Y_0:Z_0)\in E[3]$ with $v(X_0),v(Y_0)>0$ and $v(Z_0)=0$. For $p\ne3$ we see that $v(u)=0$, so the change of coordinates above does not change these valuations. For $p=3$, $v(u)=1$ and we see that the change of coordinates above increases the valuations of $X_0, Y_0$ and leaves the valuation of $Z_0$ the same.
    
    By the description of the canonical subgroup in Remark \ref{rmk:canonical}, we have shown that $E[3]\setminus\mathcal{C}_{E/K}[3]$ maps to $H[3]\setminus\mathcal{C}_{H/K}[3]$ and so conclude.
\end{proof}

Combining the above with Theorem \ref{thm:MazurRubin}, we have shown the following proposition, which we apply in the sequel to produce our families of $3$-Selmer companions:

\begin{proposition}\label{prop:3Comps}
    Suppose $E/k$ is an elliptic curve over a number field with potentially multiplicative reduction at the primes above $3$. Suppose further that $H=\Hess(E)$ has the same bad primes and the same potentially multiplicative primes as $E$, and that both curves avoid Kodaira types \emph{II}, \emph{II}$^*$, \emph{IV} and \emph{IV}$^*$ at potentially good primes. 
    
    Then $E$ and $H$ are $3$-Selmer companions over every finite extension of $k$.
\end{proposition}

\section{Constructing families}

We now come to constructing our families of $3$-Selmer companions. We will be interested in the following two-parameter families of elliptic curves:
\begin{eqnarray}
    E_{q,t}/\QQ : \hspace{0.2cm} y^2 &=& x^3 + q^2x^2 + 3q^3x +3q^6t \\
    H_{q,t}/\QQ : \hspace{0.2cm} y^2 &=& x^3+q(27-2q-81qt)x^2+q(q-9)^3x,
\end{eqnarray}
where $H_{q,t}$ is the Hessian curve associated to $E_{q,t}$. Our families of $3$-companions will consist of pairs $(E_{q,t},H_{q,t})$ for a fixed value of $q$ as $t$ is allowed to vary.

We first check that the curves' Kodaira types at $2$ avoid those forbidden by the conditions of Theorem \ref{thm:MazurRubin}.

\begin{lemma}\label{lem:kodaira_check}
    For $q\equiv3\pmod{4}$ and $t\equiv3\pmod{4}$, or $q\equiv1,5,13\pmod{16}$ and $t\equiv1\pmod{8}$, or $q\in\{8,10,12\}$, the curves $E_{q,t}$ and $H_{q,t}$ avoid Kodaira types \emph{II}, \emph{II}$^*$, \emph{IV} and \emph{IV}$^*$ at $2$.
\end{lemma}

\begin{proof}
    To prove this we apply \cite[Theorem 1]{KodairaTypesD2}, which uses Tate's algorithm to give criteria (primarily) in terms of valuations of Weierstrass coefficients for the Kodaira types of an elliptic curve at primes of additive reduction. We write $v_2$ throughout for the normalised valuation on $\QQ_2$.

    In the case $q\equiv t \equiv 3 \pmod{4}$, we make a change of variables $x'=x+q$ to obtain a Weierstrass equation for $E_{q,t}$ with $a_1=a_3=0$, $a_2=q^2-3q$, $a_4=q^3+3q^2$ and $a_6=3q^6t-2q^4-q^3$. We now see that $v_2(a_2)/2=v_2(q-3)/2\ge1$, that $v_2(a_4)/4=v_2(q+3)/4=1/4$ and that $v_2(a_6)/6=v_2(3q^3t-2q-1)/6 \ge 1/3$. We see that $\min_i v_2(a_i)/i=1/4$, so \cite[Theorem 1]{KodairaTypesD2} implies that $E_{q,t}$ has Kodaira type III at $2$. In the same case, for $H_{q,t}$ we have $v_2(a_2)/2=v_2(27-2q-81qt)/2\ge1$ and $v_2(a_4)/4=3v_2(q-9)/4=3/4$, so the Kodaira type is III$^*$.

    Now consider the case $q=8$. For $E_{q,t}$, we can rescale so that $a_2=4$, $a_4=6$ and $a_6=192t$, from which one we read off that $\min_iv_2(a_i)/i=1/4$ and so the Kodaira type is III. For $H_{q,t}$, we have $v_2(a_2)/2=v_2(88-5184t)/2=3/2$ and $v_2(a_4)/4=3/4$ so the Kodaira type is III$^*$. For $q=10$ we needn't rescale, and for $E_{q,t}$ we have $v_2(a_2)/2=v_2(100)/2=1$, $v_2(a_4)/4=v_2(3000)/4=3/4$ and $v_2(a_6)/6=v_2(3000000t)/6\ge1$ so the Kodaira type is III$^*$. For $H_{q,t}$ we have $v_2(a_2)/2=v_2(10(5-810t))/2=1/2$ and $v_2(a_4)/4=v_2(10)/4=1/4$ so the Kodaira type is III. For $q=12$, we rescale the Weierstrass coefficients of $E_{q,t}$ for $a_2=4$, $a_4=4$ and $a_6=2^6\cdot3t$ so $\min_iv_2(a_i)/i=1/2$. We have $v_2(a_2^2-3a_4)=v_2(4)=2$ and $v_2(\disc(x^3+4x^2+4x+2^6\cdot3t))=v_2(-995328t^2 + 6144t)\ge11$, from which we can conclude that the Kodaira type is I$_n^*$, for some $n$. For $H_{q,t}$ we have $v_2(a_2)/2=v_2(12(1-972t))/2=1$ and $v_2(a_4)/4=v_2(12\cdot3^3)/4=1/2$. Again, we check that $v_2(a_2^2-3a_4)=v_2(136048896t^2 - 279936t - 828)=2$ and that $v_2(\disc(x^3+12(1-972t)x^2+12\cdot3^3x))=v_2(14281868906496t^2 - 29386561536t - 120932352)\ge11$, so again the Kodaira type is I$_n^*$.
    
    For the cases $q\equiv1,5,13\pmod{16}$ and $t\equiv1\pmod{8}$ one needs to subdivide the cases further, distinguishing between $t\equiv1\pmod{16}$ and $t\equiv 9\pmod{16}$. For the sake of brevity we give slightly less detail here, but explain how to find the relevant Kodaira types via \cite[Theorem 1]{KodairaTypesD2} in all of these cases.

    Firstly we treat the curves $E_{q,t}$. Suppose $q\equiv 1\pmod{16}$. After the same change of variables as above, we see that for $E_{q,t}$ we have $\min_iv_2(a_i)/i=1/2$ regardless of whether $t\equiv1\pmod{16}$ or $t\equiv 9\pmod{16}$. For $t\equiv1\pmod{16}$, we compute $d:=\disc(x^3+a_2x^2+a_4x+a_6)$, write $q=16q_0+1$, $t=16t_0+1$ and substitute this into $d$, taking coefficients modulo $2^7=128$. This shows that $v_2(d)=6$ and so we conclude that the Kodaira type is I$_0^*$. Precisely the same argument gives the same Kodaira type for $q\equiv 5\pmod{16}$, $t\equiv 1\pmod{16}$ and $q\equiv 13\pmod{16}$, $t\equiv 9\pmod{16}$. For $q\equiv5\pmod{16}$, $t\equiv{9}\pmod{16}$ and $q\equiv13\pmod{16}$, $t\equiv 1\pmod{16}$, we follow the same argument but find that $v_2(d)\ge7$. In these cases we conclude that the Kodaira type is I$_n^*$ by noticing that $a_2^2-3a_4\equiv 4\pmod{8}$ for $q\equiv 5\pmod{8}$. For $q\equiv1\pmod{16}$ and $t\equiv 9\pmod{16}$ we note that $v_2(a_4-a_2^2/3)=3$, whilst $v_2(2a_2^3/27-a_2a_4/3+a_6)\ge4$ so completing the cube gives a model from which we read off that the Kodaira type is III$^*$.

    It remains to treat the $H_{q,t}$. Here we need to rescale the $a_i$ either by $2^{-i}$ or by $2^{-2i}$. The former is the case for $q\equiv5\pmod{8}$ (when $v_2(a_4)=6$ and $v_2(a_2)\ge2$) and $q\equiv1\pmod{16}$, $t\equiv 1\pmod{16}$ (when $v_2(a_4)=9$ and $v_2(a_2)=3$). If $v_2(a_2)=0$ after rescaling, then also apply the transformation $y'=y-x$. In all cases aside from $q\equiv1\pmod{16}$, $t\equiv9\pmod{16}$, after the changes of coordinates we see that $\min_i v_2(a_i)/i = 1/2$ and from here one can check that the Kodaira types are either I$_0^*$ or I$_n^*$ by computing the valuations of $d$ and of $a_2^2-3a_4$ as before. Lastly, in the case $q\equiv1\pmod{16}$, $t\equiv9\pmod{16}$ we have $v_2(a_2)\ge4$ and $v_2(a_4)=9$ prior to rescaling, so $v_2(a_4)=1$ after rescaling and either $v_2(a_2)\ge1$, in which case the Kodaira type is III, or $v_2(a_2)=0$ and we reach the same conclusion after the change of variables $y'=y-x$.
\end{proof}

Using the results of the preceding section, we are thus able to show:

\begin{theorem}\label{thm:main}
    In each case described by Table \ref{tab:1}, except for $(q,t)=(12,0)$, the curves $E_{q,t}$ and $H_{q,t}$ are $3$-Selmer companions over every number field.
\end{theorem}

\begin{proof}
    In the exceptional case the curves are singular. For the rest it suffices to verify the conditions of Proposition \ref{prop:3Comps}, except for checking the Kodaira types at $2$, which is done by Lemma \ref{lem:kodaira_check}. 
    
    Firstly, it is simple to see that $E_{q,t}$ has multiplicative reduction at $3$ in each of the cases to be considered. 
    
    Now note that the ratio $\Delta_{H_{q,t}}/\Delta_{E_{q,t}}$ is
    \[-27 + 1458q^{-1} - 32805q^{-2} + 393660q^{-3} - 2657205q^{-4} + 9565938q^{-5} -14348907q^{-6}\]
    and so is independent of $t$. Hence we can easily check that there are no primes at which $H_{q,t}$ has bad reduction but $E_{q,t}$ has good reduction in each case of Table \ref{tab:1}.

    We also have
    \[j(E_{q,t})=\frac{4096q^3 - 110592q^2 + 995328q - 2985984}{-3888q^3t^2 - 192q^3t + 2592q^2t + 144q - 1728}.\]
    This denominator is $\Delta_{E_{q,t}}/q^9$ and, in each of our cases, this numerator is either a power of $2$ or a power of $2$ multiplied by a power of $3$, independently of $t$. It follows that for each value of $t$, the numerator of $j(E_{q,t})$ is a power of $2$ and that the denominator differs from the discriminant only by powers of $2,3$ and $q$. Hence the only primes of potentially good reduction for $E_{q,t}$ are $2$ and primes greater than $3$ dividing $q$. The case $p=2$ is checked in Lemma \ref{lem:kodaira_check} and, because the numerator of $j(E_{q,t})$ is a power of $2$, we see that $E_{q,t}$ cannot have Kodaira type {II}, {II}$^*$, {IV} or {IV}$^*$ at such a prime $2\ne p\mid q$ (as these types can only arise when $j\equiv0\pmod{p}$). To show that $H_{q,t}$ avoids these Kodaira types we run through Tate's algorithm, following the steps as described on \cite[p. 366]{SilvermanAdvanced}:

    We pass through step 1 because it is clear that $p\mid\Delta_{H_{q,t}}$, through step 2 without needing to make a change of coordinates, and then through step 3 because $a_6=0$. Upon reaching step 4, we see that $p^3\nmid b_8$ and so the Kodaira type is {III}.
\end{proof}

\begin{remark}\label{rmk:exhaust}
    To do anything but treat the values of $q$ on a one-by-one basis seems very difficult because one must keep track of the primes dividing the ratio of the discriminants, as well as of potentially good primes and the corresponding Kodaira types. Therefore we do not expect Theorem \ref{thm:main} to be exhaustive, in the sense that there are likely values of $q$ not appearing in Table \ref{tab:1} for which some condition on $t$ will yield a family of $3$-Selmer companions.
\end{remark}

It remains to verify that the curves of Theorem \ref{thm:main} do not consist of isogenous pairs. To do so we invoke a well-known theorem of Mazur and Kenku.

\begin{theorem}[{=\cite[Theorem 1]{Kenku1982}}] \label{thm:MazurKenku}
        Suppose an elliptic curve over $\QQ$ admits a cyclic $\QQ$-isogeny of degree $N$. Then $N\le 19$ or $N\in \{21, 25, 27, 37, 43, 67, 163\}$.
\end{theorem}

Write $\Phi_N$ for the $N^\text{th}$ classical modular polynomial and recall that there is a degree $N$ cyclic $\bar{\QQ}$-isogeny between two rational elliptic curves $E$ and $H$ if and only if $\Phi_N(j(E),j(H))=0$ (e.g.\ \cite[II, Exercise 2.19(b)]{SilvermanAdvanced}).

\begin{proposition}\label{prop:isog}
    For each value of $q$ in Table \ref{tab:1} and $t\in\ZZ$ the curves $E_{q,t}$ and $H_{q,t}$ are not $\bar{\QQ}$-isogenous, provided $(q,t)\notin\{(1,1),(1,9),(3,-1),(3,0),(8,0),(12,0)\}$.
\end{proposition}

A key point in the proof of the above is the following lemma:

\begin{lemma}\label{lem:Q_isog}
    Suppose $E$, $H$ are elliptic curves over $\QQ$  which do not have potential complex multiplication. Possibly after replacing one of the curves by a quadratic twist, any isogeny $\phi: E\to H$ can be taken to be defined over $\QQ$.
\end{lemma}

\begin{proof}
    Take $\phi: E \to H$ defined over a proper extension of $\QQ$ with Galois conjugate $\phi^\sigma \ne \phi$, so $(\phi^\sigma)^\vee\circ\phi \in \End(E)=\ZZ$, from whence $\phi^\sigma = \pm \phi$ by comparing degrees.
    
    Take $E$, $H$ to be given by short Weierstrass models. Then for any $\phi$, writing $\phi=(x_\phi,y_\phi)$, we see that $x_\phi$ is a rational function and at worst Galois acts by $y_\phi\mapsto -y_\phi$. Therefore the Galois action is trivial or factors through a quadratic field and $y_\phi$ is rational or defined over $\QQ(\sqrt{d})$, for some $d$. In the latter case, $y_\phi\sqrt{d}$ is fixed by Galois and hence is rational.
\end{proof}

\begin{proof}[Proof of Proposition \ref{prop:isog}]
    Firstly we note that the $j$-invariants of $E_{q,t}$, $H_{q,t}$ are non-integral in all the cases we consider (e.g.\ because both have multiplicative reduction at $3$) so neither curve has potential CM, and we are in the situation of Lemma \ref{lem:Q_isog}.
    
    Recall that the ratio $\Delta_{H_{q,t}}/\Delta_{E_{q,t}}$ is independent of $t$, so a quick check shows that this ratio is $-3$ modulo squares for all pairs $(q,t)$ in which we are interested. Therefore the cubics $x^3+q^2x^2+3q^3x+3q^6t$ and $x^3+q(27-2q-81qt)x^2+q(q-9)^3x$ have different splitting fields (hence the curves have different $2$-torsion fields) for all values of $t$ and so our elliptic curves can never be isogenous via an isogeny of odd degree, as this would give an isomorphism on $2$-torsion. Quadratic twists have isomorphic $2$-torsion so this is also the case when replacing one of $E_{q,t}$ or $H_{q,t}$ by a quadratic twist.

    Now we need only check for isogenies between curves with $j$-invariants $j(E_{q,t})$ and $j(H_{q,t})$ of even degrees appearing in Theorem \ref{thm:MazurKenku}. To do so, fixing in turn each relevant value of $q$, we compute $\Phi_N(j(E_{q,t}),j(H_{q,t}))$ as a rational function in $t$ for each of these values of $N$ and compute the rational roots of its numerator. The integer roots we find are precisely those excluded in the statement of the proposition (except for the degenerate case $(q,t)=(12,0)$).
    
     Note that the relevant modular polynomials are stored in Magma \cite{Magma}, which we use to carry out these computations, except for $\Phi_{18}$, which is computed by the authors of \cite{BOS2016} and is available at \url{https://math.mit.edu/~drew/ClassicalModPolys.html}. To allow the reader to verify this proof, we have made the code for this computation available at \url{https://github.com/hspen99/HessianIsogenyComp}.
\end{proof}

\noindent\textbf{Acknowledgements.} I thank Vladimir Dokchitser for his supervision and comments on an earlier draft, Dominik Bullach for suggesting that I look into Selmer companions, and Alexandros Konstantinou for a helpful conversation, including for pointing me towards the reference \cite{MPT24}. I thank the anonymous referee for their many helpful comments and suggestions, not least for their ideas towards the given, more elegant, proofs of Lemma \ref{lem:mult} and Proposition \ref{prop:canon}.

This work was supported by the Engineering and Physical Sciences Research Council [EP/S021590/1], the EPSRC Centre for Doctoral Training in Geometry and Number Theory (The London School of Geometry and Number Theory), University College, London.

\bibliographystyle{plain}
\bibliography{main} 

\begin{thebibliography}{1}

\bibitem{Magma}
Wieb Bosma, John Cannon, and Catherine Playoust.
\newblock {The Magma Algebra System I: The User Language}.
\newblock {\em J. Symb. Comput.}, 24(3):235--265, 1997.

\bibitem{BOS2016}
Jan~Hendrik Bruinier, Ken Ono, and Andrew~V. Sutherland.
\newblock Class polynomials for nonholomorphic modular functions.
\newblock {\em J. Number Theory}, 161:204--229, 2016.

\bibitem{KodairaTypesD2}
Tim Dokchitser and Vladimir Dokchitser.
\newblock A remark on {T}ate’s algorithm and {K}odaira types.
\newblock {\em Acta Arith.}, 160(1):95--100, 2013.

\bibitem{Kenku1982}
Monsur~A. Kenku.
\newblock On the number of $\mathbb{Q}$-isomorphism classes of elliptic curves in each $\mathbb{Q}$-isogeny class.
\newblock {\em J. Number Theory}, 15(2):199--202, 1982.

\bibitem{MazurRubin2012}
Barry Mazur and Karl Rubin.
\newblock Selmer companion curves.
\newblock {\em Trans. Amer. Math. Soc.}, 367(1):401--421, 2015.

\bibitem{MPT24}
Marzio Mula, Federico Pintore, and Daniele Taufer.
\newblock The {Hessian} of elliptic curves.
\newblock arXiv: \href{https://arxiv.org/abs/2407.17042}{\url{2407.17042}}, 2024.

\bibitem{SilvermanAdvanced}
Joseph~H. Silverman.
\newblock {\em Advanced topics in the arithmetic of elliptic curves}, volume 151 of {\em GTM}.
\newblock Springer--Verlag, 1994.

\bibitem{Yu2Companions}
Myungjun Yu.
\newblock 2-{S}elmer near-companion curves.
\newblock {\em Trans. Amer. Math. Soc.}, 372(1):425--440, 2019.

\end{thebibliography}

\end{document}